\documentclass[11pt,final,a4paper]{article}%
\usepackage{titlesec,titletoc}
\usepackage[titletoc,toc,title]{appendix}
\usepackage{wrapfig}
\usepackage{amscd}
\usepackage[overload]{textcase}
\usepackage{makeidx}
\usepackage{etex}
\usepackage{amsxtra}
\usepackage{graphicx}
\usepackage[amsmath,hyperref,thmmarks]{ntheorem}
\usepackage{natbib}
\usepackage{amsmath}
\usepackage{amsfonts}
\usepackage{amssymb}
\usepackage{stix2}
\usepackage{natbib}
\usepackage[colorlinks=true,bookmarksnumbered=true,pdfpagemode=None]{hyperref}%
\providecommand{\U}[1]{\protect\rule{.1in}{.1in}}
\theoremnumbering{arabic}
\theoremheaderfont{\scshape}
\RequirePackage{latexsym}
\theorembodyfont{\slshape}
\theoremseparator{.}
\newtheorem{X}{X}[section]

\newtheorem{corollary}[X]{Corollary}

\newtheorem{lemma}[X]{Lemma}
\newtheorem{proposition}[X]{Proposition}

\newtheorem{theorem}[X]{Theorem}
\newtheorem{maintheorem}{Theorem}
\theorembodyfont{\upshape}

\newtheorem{plain}[X]{}

\newtheorem{remark}[X]{Remark}

\theorembodyfont{\small}
\theorembodyfont{\normalsize}
\theoremstyle{nonumberplain}
\theoremsymbol{\ensuremath{_\Box}}
\newtheorem{proof}{Proof}

\qedsymbol{\ensuremath{_\Box}}
\newcommand{\df}{\smash{\lower.12em\hbox{\textup{\tiny def}}}}

\newcommand{\mss}{\mathscr{s}}

\contentsmargin{2.0em}
\dottedcontents{section}[2.3em]{}{1.8em}{1pc}
\dottedcontents{subsection}[6.0em]{}{1.8em}{1pc}
\titleformat*{\subsection}{\large\scshape}
\titleformat*{\subsubsection}{\slshape}

\usepackage{titlesec}                                                                            
\titleformat*{\section}{\LARGE\bfseries}
\titleformat*{\subsection}{\Large\itshape}
\titleformat*{\subsubsection}{\scshape}%\large
\titleformat*{\paragraph}{\itshape}

\usepackage{enumitem}
\setitemize[1]{label=$\diamond$\hspace{0.07in}}
\setlist{nolistsep}
\usepackage[nottoc]{tocbibind}

\usepackage{array}

\usepackage{microtype}

\usepackage{tikz}
\usetikzlibrary{matrix,arrows,positioning,decorations.pathmorphing}
\usepackage{tikz-cd}
\tikzcdset{arrow style=math font}
\usepackage{wrapfig}

\usepackage{etex}

\usepackage{url}
\usepackage{verbatim}

\usepackage[notcite,notref]{showkeys}

\usepackage{mathtools}

\usepackage[overload]{textcase}

\newcommand{\eb}[1]{{\itshape\bfseries#1}}
\renewcommand{\emph}{\eb}

\usepackage{natbib}
\let\cite\citealt
\hypersetup{linkcolor=blue,anchorcolor=blue,citecolor=blue,filecolor=blue,urlcolor=blue}
\hypersetup{pdftitle={Consequences of the Lefschetz standard conjecture},pdfauthor={J.S. Milne}}
\hypersetup{pdfsubject={standard conjectures},pdfkeywords={Lefschetz standard conjecture, Tate conjecture}}

\usepackage[textwidth=5.5in,textheight=9.0in,centering]{geometry}
\newcommand{\bcomment}{\begin{comment}}\newcommand\ecomment{\end{comment}}
\newcommand{\bfootnotesize}{\begin{footnotesize}}\newcommand\efootnotesize{\end{footnotesize}}
\newcommand{\bquote}{\begin{quote}}\newcommand\equote{\end{quote}}
\newcommand{\bsmall}{\begin{small}}\newcommand\esmall{\end{small}}
\newcommand{\btable}{\begin{table}}\newcommand{\etable}{\end{table}}
\newcommand{\bappendices}{\begin{appendices}}
\newcommand{\eappendices}{\end{appendices}}

\newcommand{\edocument}{\end{document}}

\newcommand{\tableoc}{\tableofcontents}%* removes "contents" from TOC

\renewcommand{\Gamma}{\varGamma}
\renewcommand{\Pi}{\varPi}
\renewcommand{\Sigma}{\varSigma}
\def\1{{1\mkern-7mu1}}

\DeclareMathOperator{\End}{End}

\DeclareMathOperator{\Gal}{Gal}

\DeclareMathOperator{\Hom}{Hom}
\DeclareMathOperator{\id}{id}

\DeclareMathOperator{\Spec}{Spec}

\newcommand{\Hdg}{\mathsf{Hdg}}%Hodge structures

\newcommand{\Mot}{\mathsf{Mot}}%motives

\newcommand{\Vc}{\mathsf{Vec}}%%\Vec would be better, but confuses SWP

\makeindex
\begin{document}

\title{Grothendieck's standard conjecture of Lefschetz type over finite fields}
\author{J. S. Milne}
\date{\today}
\maketitle

\begin{abstract}
Grothendieck's standard conjecture of Lefschetz type has two main forms: the
weak form $C$ and the strong form $B$. The weak form is known for varieties
over finite fields as a consequence of the proof of the Weil conjectures. This
suggests that the strong form of the conjecture in the same setting may be the
most accessible of the standard conjectures. Here, as an advertisement for the
conjecture, we explain some of its remarkable consequences.

\end{abstract}

\tableoc\bigskip

All algebraic varieties are projective, smooth, and connected (unless denoted
by $S$). Let $H$ be a Weil cohomology theory on the algebraic varieties over
an algebraically closed field $k$. Let $X$ be a variety over $k$ of dimension
$d$, and let $L\colon H^{\ast}(X)\rightarrow H^{\ast+2}(X)$ be the Lefschetz
operator defined by a hyperplane section. The strong Lefschetz theorem states
that%
\[
L^{d-i}\colon H^{i}(X)\rightarrow H^{2d-i}(X)
\]
is an isomorphism for all $i\leq d$. The Lefschetz standard conjecture (in its
strong form) states that $L^{d-2i}$ induces an isomorphism on the
$\mathbb{Q}{}$-subspaces of algebraic classes (see \cite{kleiman1968}). We say
that the Lefschetz standard conjecture holds for the algebraic varieties over
a field $k$ if the strong form holds for the classical Weil cohomology
theories $(\ell$-adic \'{e}tale, de Rham in characteristic zero, crystalline
in characteristic $p$).

An \textit{almost-algebraic class} on an algebraic variety in characteristic
zero is an absolute Hodge class that becomes algebraic modulo $p$ for almost
all $p$ (see \ref{g6} below).

\begin{maintheorem}
\label{g1}The Lefschetz standard conjecture for algebraic varieties over
finite fields implies the almost-Hodge conjecture for abelian varieties, i.e.,
all Hodge classes on complex abelian varieties are almost-algebraic.
\end{maintheorem}

In the remaining statements, $p$ is a fixed prime number and $\mathbb{F}{}$ is
an algebraic closure of the field $\mathbb{F}{}_{p}$ of $p$ elements.

\begin{maintheorem}
\label{g3}The Lefschetz standard conjecture for algebraic varieties over
$\mathbb{F}{}$ implies

\begin{enumerate}
\item the full Tate conjecture for abelian varieties over $\mathbb{F}$;

\item the standard conjecture of Hodge type for abelian varieties in
characteristic $p$.
\end{enumerate}
\end{maintheorem}

See \S 3 for the Tate conjecture and \cite{kleiman1994}, p.~16, for the Hodge
standard conjecture.

\begin{maintheorem}
\label{g4}The full Tate conjecture for algebraic varieties over $\mathbb{F}$
implies Grothendieck's standard conjectures over $\mathbb{F}{}$.
\end{maintheorem}

We prove these theorems in the first four sections of the paper. In Section 5,
we use a construction of Sch\"{a}ppi to give unconditional variants of our
theorems, and in Section 6, we list some statements that imply the Lefschetz
standard conjecture.

We refer to \cite{kleiman1994} for the various forms, $A,B,C,D,$ of the
Lefschetz standard conjecture. We assume that the reader is familiar with the
expository article \cite{milneHAV}, cited as HAV. Throughout, $\mathbb{Q}^{\mathrm{al}}$ is the algebraic closure of $\mathbb{Q}{}$ in $\mathbb{C}{}$.

\section{Proof of Theorem 1}

In this section, $X$ is an algebraic variety over a field $k$ of
characteristic zero.

\begin{plain}
\label{g2}Suppose first that $k$ is algebraically closed. We let $B^{r}(X)$
denote the space of absolute Hodge classes of codimension $r$ on $X$. Thus
$B^{r}(X)$ is a finite-dimensional $\mathbb{Q}{}$-subspace of the ad\`{e}lic
cohomology group $H_{\mathbb{A}}^{2r}(X)(r)$ (\cite{deligne1982}, \S 3). Let
$k\rightarrow k^{\prime}$ be a homomorphism from $k$ into a second
algebraically closed field $k^{\prime}$; then the canonical map $H_{\mathbb{A}%
}^{2r}(X)(r)\rightarrow H_{\mathbb{A}{}}^{2r}(X_{k^{\prime}})(r)$ induces an
isomorphism $B^{r}(X)\rightarrow B^{r}(X_{k^{\prime}})$ (ibid., 2.9). For an
abelian variety over $\mathbb{C}$, every Hodge class is absolutely Hodge
(ibid., Main Theorem 2.11).
\end{plain}

\begin{plain}
\label{g5}Now allow $k$ to be arbitrary of characteristic $0$, and let
$k^{\mathrm{al}}$ be an algebraic closure of $k$. Then $\Gal(k^{\mathrm{al}%
}/k)$ acts on $B^{r}(X_{k^{\mathrm{al}}})$ through a finite quotient, and
$B^{r}(X)\overset{\df}{=}B^{r}(X_{k^{\mathrm{al}}})^{\Gal(k^{\mathrm{al}}/k)}$.
\end{plain}

\begin{plain}
\label{g6}We define an \emph{almost-algebraic} class of codimension $r$ on $X$
to be an absolute Hodge class $\gamma$ of codimension $r$ such that there
exists a cartesian square%
\[
\begin{tikzcd}
\mathcal{X}\arrow{d}{f}& X\arrow{l}\arrow{d}\\
S & \Spec(k)\arrow{l}
\end{tikzcd}
\]
and a global section $\tilde{\gamma}$ of $R^{2r}f_{\ast}\mathbb{A}{}(r)$
satisfying the following conditions,

\begin{itemize}
\item $S$ is the spectrum of a regular integral domain of finite type over
$\mathbb{Z}{};$

\item $f$ is smooth and projective;

\item the fibre of $\tilde{\gamma}$ over $\Spec(k)$ is $\gamma$, and the
reduction of $\tilde{\gamma}$ at $s$ is algebraic for all closed points $s$ in
a dense open subset of $S$.
\end{itemize}

\noindent Cf. \cite{serre1974}, 5.2, and \cite{tate1994}, p.76. Usually,
almost-algebraic classes are not required to be absolutely Hodge, but since we
have a robust theory of absolute Hodge classes, it is natural to include it.
\end{plain}

\begin{theorem}
\label{g7}Assume that the Lefschetz standard conjecture holds for algebraic
varieties over finite fields. Then all absolute Hodge classes on abelian
varieties over fields of characteristic zero are almost-algebraic.
\end{theorem}

\begin{proof}
It suffices to prove this with $k=\mathbb{C}{}$, where it becomes a question
of showing that Hodge classes on abelian varieties are almost-algebraic. Let
$X$ be an algebraic variety of dimension $d$ over $\mathbb{C}{}$, and let
$L\colon H^{\ast}(X,\mathbb{Q}{})\rightarrow H^{\ast+2}(X,\mathbb{Q}{}(1))$ be
the Lefschetz operator on Betti cohomology defined by a hyperplane section.
According to the strong Lefschetz theorem, the map $L^{d-i}\colon
H^{i}(X,\mathbb{Q})\rightarrow H^{2d-i}(X,\mathbb{Q}{})(d-i)$ is an
isomorphism. For $i\leq d$, let $\theta^{i}\colon H^{2d-i}(X,\mathbb{Q}%
{})(d-i)\rightarrow H^{i}(X,\mathbb{Q}{})$ denote the inverse isomorphism.

The isomorphism $\theta^{i}\otimes1\colon H_{\mathbb{A}{}}^{2d-i}%
(X)(d-i)\rightarrow H_{\mathbb{A}}^{i}(X)$ is absolutely Hodge (i.e., its
graph is an absolute Hodge class). Consider a diagram as in \ref{g6}. For a
closed point $s$ of $S$ such that $X$ and $L$ have good reduction, $\theta
^{i}\otimes1$ specializes to the inverse of the isomorphism $L(s)^{d-i}\colon
H_{\mathbb{A}{}}^{i}(X(s))\rightarrow H_{\mathbb{A}{}}^{2d-i}(X(s))(d-i)$. As
we are assuming the standard conjecture over $\mathbb{F}{}$, this inverse is
algebraic (\cite{kleiman1994}, 4-1, $\theta\Leftrightarrow B$). Hence
$\theta^{i}$ is almost-algebraic.

Since this holds for all $X$ and $i$, the Lefschetz standard conjecture holds
for almost-algebraic classes on algebraic varieties over $\mathbb{C}{}$. As
for algebraic classes, this implies that all Hodge classes on abelian
varieties are almost-algebraic (HAV, Theorem 4).
\end{proof}

Note that the theorem does not say that an absolute Hodge class becomes
algebraic modulo $p$ for any specific $p$, even when the abelian variety has
good reduction at $p$. In the next section, we prove this.

\section{Almost-algebraic classes on abelian varieties}

Fix a prime number $p$, and let $\mathbb{F}{}$ be an algebraic closure of
$\mathbb{F}{}_{p}$. In the following, $\ell$ is a prime number $\neq p$.

\subsection{Variation of algebraic classes over $\mathbb{F}{}$}

\begin{proposition}
\label{g13}Let $S$ be a complete smooth curve over $\mathbb{F}{}$ and $f\colon
X\rightarrow S$ an abelian scheme over $S$. Assume that the Lefschetz standard
conjecture holds for $X$ and $\ell$-adic \'{e}tale cohomology. Let $t$ be a
global section of the sheaf $R^{2r}f_{\ast}\mathbb{Q}{}_{\ell}(r)$; if $t_{s}$
is algebraic for one $s\in S(\mathbb{F})$, then it is algebraic for all $s$.
\end{proposition}

\begin{proof}
For a positive integer $n$ prime to $p$, let $\theta_{n}$ denote the
endomorphism of $X/S$ acting as multiplication by $n$ on the fibres. By a
standard argument (\cite{kleiman1968}, p.~374), $\theta_{n}^{\ast}$ acts as
$n^{j}$ on $R^{j}f_{\ast}\mathbb{Q}_{\ell}$. As $\theta_{n}^{\ast}$ commutes
with the differentials $d_{2}$ of the Leray spectral sequence $H^{i}%
(S,R^{j}f_{\ast}\mathbb{Q}_{\ell})\implies H^{i+j}(X,\mathbb{Q}{}_{\ell})$, we
see that it degenerates at the $E_{2}$-term and%
\[
H^{r}(X,\mathbb{Q}_{\ell}\mathbb{)}\simeq\bigoplus_{i+j=r}H^{i}(S,R^{j}%
f_{\ast}\mathbb{Q}{}_{\ell})\text{,}%
\]
where $H^{i}(S,R^{j}f_{\ast}\mathbb{Q}{}_{\ell})$ is the direct summand of
$H^{r}(X,\mathbb{Q}{}_{\ell}\mathbb{)}$ on which $n$ acts as $n^{j}$. We let
$a\!H$ denote the $\mathbb{Q}{}$-subspace of a cohomology group $H$ spanned by
the algebraic classes.

Let $s\in S(\mathbb{F}{})$ and let $\pi=\pi_{1}(S,s)$. The inclusion
$j_{s}\colon X_{s}\hookrightarrow X$ induces an isomorphism $j_{s}^{\ast
}\colon H^{0}(S,R^{2r}f_{\ast}\mathbb{Q}{}_{\ell})\rightarrow H^{2r}%
(X_{s},\mathbb{Q}{}_{\ell})^{\pi}$ preserving algebraic classes, and so%
\begin{equation}
\dim a\!H^{0}(S,R^{2r}f_{\ast}\mathbb{Q}{}_{\ell})\leq\dim a\!H^{2r}%
(X_{s},\mathbb{Q}{}_{\ell})^{\pi}. \label{e2}%
\end{equation}
Similarly, the Gysin map $j_{s\ast}\colon H^{2d-2r}(X_{s},\mathbb{Q}{}_{\ell
})\rightarrow{}H^{2d-2r+2}(X,\mathbb{Q}{}_{\ell})$, where $d=\dim(X/S)$,
induces a map $H^{2d-2r}(X_{s},\mathbb{Q}{}_{\ell})^{\pi}\rightarrow
H^{2}(S,R^{2d-2r}f_{\ast}\mathbb{Q}_{\ell})$ preserving algebraic classes, and
so%
\begin{equation}
\dim a\!H^{2d-2r}(X_{s},\mathbb{Q}{}_{\ell})^{\pi}\leq\dim a\!H^{2}%
(S,R^{2d-2r}f_{\ast}\mathbb{Q}_{\ell})\text{.} \label{e3}%
\end{equation}
Because the Lefschetz standard conjecture holds for $X_{s}$
(\cite{kleiman1968}, 2A11),%
\begin{equation}
\dim a\!H^{2r}(X_{s},\mathbb{Q}{}_{\ell})^{\pi}=\dim a\!H^{2d-2r}%
(X_{s},\mathbb{Q}{}_{\ell})^{\pi}\text{.} \label{e4}%
\end{equation}
Hence,
\begin{align*}
\dim a\!H^{0}(S,R^{2r}f_{\ast}\mathbb{Q}_{\ell}{})  &
\overset{(\text{\ref{e2}})}{\leq}\dim a\!H^{2r}(X_{s}{},\mathbb{Q}{}_{\ell
})^{\pi}\overset{(\text{\ref{e4})}}{=}\dim a\!H^{2d-2r}(X_{s},\mathbb{Q}%
{}_{\ell})^{\pi}\\
&  \overset{(\text{\ref{e3}})}{\leq}\dim a\!H^{2}(S,R^{2d-2r}f_{\ast
}\mathbb{Q}_{\ell})\text{.}%
\end{align*}
The Lefschetz standard conjecture for $X$ implies that
\[
\dim a\!H^{0}(S,R^{2r}f_{\ast}\mathbb{Q}{}_{\ell})=\dim a\!H^{2}%
(S,R^{2d-2r}f_{\ast}\mathbb{Q}_{\ell}),
\]
and so the inequalities are equalities. Thus%
\[
a\!H^{2r}(X_{s},\mathbb{Q}{}_{\ell})^{\pi}=a\!H^{0}(S,R^{2r}f_{\ast}%
\mathbb{Q}_{\ell}),
\]
which is independent of $s$.
\end{proof}

\begin{remark}
\label{g14}The proof shows that $t$, when regarded as an element of
$H^{2r}(X,\mathbb{Q}{}_{\ell}(r))$, is algebraic.
\end{remark}

\subsection{Weil classes\qquad}

Fix a prime $w$ of $\mathbb{Q}{}^{\mathrm{al}}$ dividing $p{}$. The residue
field at $w$ is an algebraic closure $\mathbb{F}{}$ of $\mathbb{F}{}_{p}$. We
refer to \cite{deligne1982} or HAV for facts on abelian varieties of Weil type.

\begin{proposition}
\label{g8}Assume that the Lefschetz standard conjecture holds for algebraic
varieties over $\mathbb{F}{}$ and $\ell$-adic \'{e}tale cohomology, some
$\ell\neq p$. Let $(A$, $\nu)$ be an abelian variety over $\mathbb{Q}%
{}^{\mathrm{al}}$ of split Weil type relative to a CM field $E$, and let $t\in
W_{E}(A)\subset H_{\mathbb{A}}^{2r}(A)$ be a Weil class on $A$. If $A$ has
good reduction at $w$ to an abelian variety $A_{0}$ over $\mathbb{F}{}$, then
the element $(t_{\ell})_{0}$ of $H^{2r}(A_{0},\mathbb{Q}{}_{\ell})$ is algebraic.
\end{proposition}

The proof will occupy the remainder of this subsection. In outline, it follows
the proof of \cite{deligne1982}, Theorem 4.8, but requires a delicate
reduction argument of Andr\'{e}.

\begin{lemma}
\label{g11}Let $(A,\nu)$ be an abelian variety over $\mathbb{\mathbb{Q}{}%
}^{\mathrm{al}}$ of split Weil type relative to $E$. Then there exists a
connected smooth variety $S$ over $\mathbb{C}{}$, an abelian scheme $f\colon
X\rightarrow S$ over $S$, and an action $\nu$ of $E$ on $X/S$ such that

\begin{enumerate}
\item for some $s_{1}\in S(\mathbb{C}{})$, $(X_{s_{1}},\nu_{s_{1}}%
)\approx(A,\nu)_{\mathbb{C}{}};$

\item for all $s\in S(\mathbb{C}{})$, $(X_{s},\nu_{s})$ is of split Weil type
relative to $E$;

\item for some $s_{2}\in S(\mathbb{C}{})$, $X_{s_{2}}$ is of the form
$B\otimes_{\mathbb{Q}{}}E$ with $e\in E$ acting as $\id\otimes e$.
\end{enumerate}
\end{lemma}

\begin{proof}
See the proof of \cite{deligne1982}, 4.8.
\end{proof}

We shall need to use additional properties of the family $X\rightarrow S$
constructed by Deligne. For example, there is a local subsystem $W_{E}(X/S)$
of $R^{2r}f_{\ast}\mathbb{Q}{}$ such that $W_{E}(X/S)_{s}=W_{E}(X_{s})$ for
all $s\in S(\mathbb{C}{})$. Also, the variety $B$ in (c) can be chosen to be a
power of CM elliptic curve (so $X_{s_{2}}$ is isogenous to a power of a CM
elliptic curve).

The variety $S$ has a unique model over $\mathbb{Q}{}^{\mathrm{al}}$ with the
property that every CM-point $s\in S(\mathbb{C}{})$ lies in $S(\mathbb{Q}%
{}^{\mathrm{al}})$. This follows from the general theory of Shimura varieties;
or from the general theory of locally symmetric varieties (Faltings, Peters);
or (best) from descent theory (\cite{milne1999d}, 2.3) using that $S$ is a
moduli variety over $\mathbb{C}{}$ and that the moduli problem is defined over
$\mathbb{Q}{}^{\mathrm{al}}$. The morphism $f$ is also defined over
$\mathbb{Q}{}^{\mathrm{al}}$, and we will now simply write $f\colon
X\rightarrow S$ for the family over $\mathbb{Q}{}^{\mathrm{al}}$. There is a
$\mathbb{Q}{}$-local subsystem $W_{E}(X/S)$ of $R^{2r}f_{\ast}\mathbb{Q}%
{}_{\ell}$ such that $W_{E}(X/S)_{s}=W_{E}(X_{s})$ for all $s\in
S(\mathbb{Q}{}^{\mathrm{al}})$. The points $s_{1}$ and $s_{2}$ lie in
$S(\mathbb{Q}{}^{\mathrm{al}})$.

We now assume that $E$ contains an imaginary quadratic field in which the
prime $p$ splits --- this is the only case we shall need, and it implies the
general case.

The family $X\rightarrow S$ (without the action of $E$) defines a morphism
from $S$ into a moduli variety $M$ over $\mathbb{Q}{}^{\mathrm{al}}$ for
polarized abelian varieties with certain level structures. Let $\mathcal{M}{}$
be the corresponding moduli scheme over $\mathcal{O}{}_{w}$ and $\mathcal{M}%
{}^{\ast}$ its minimal compactification (\cite{chaiF1990}). Let $\mathcal{S}%
{}^{\ast}$ be the closure of $S$ in $\mathcal{M}{}^{\ast}$.

\begin{lemma}
\label{g12}The complement of $\mathcal{S}_{\mathbb{F}{}}^{\ast}\cap
\mathcal{M}{}_{\mathbb{F}{}}$ in $\mathcal{S}{}_{\mathbb{F}{}}^{\ast}$ has
codimension at least two.
\end{lemma}

\begin{proof}
See Andr\'{e} 2.4.2.
\end{proof}

Recall that $s_{1}$ and $s_{2}$ are points in $S(\mathbb{Q}{}^{\mathrm{al}})$
such that $X_{s_{1}}=A$ and $X_{s_{2}}$ is a power of a CM-elliptic curve. As
$A$ and the elliptic curve have good reduction, the points extend to points
$\mss_{1}$ and $\mss_{2}$ of $\mathcal{S}{}^{\ast}\cap\mathcal{\mathcal{M}{}}%
$. Let $\mathcal{\bar{S}}$ denote the blow-up of $\mathcal{S}{}^{\ast}$
centred at the closed subscheme defined by the image of $\mss_{1}$ and
$\mss_{2}$, and let $\mathcal{S}{}$ be the open subscheme obtained by removing
the strict transform of the boundary $\mathcal{S}{}^{\ast}\smallsetminus
(\mathcal{S}{}^{\ast}\cap\mathcal{M}{})$. It follows from \ref{g12} that
$\mathcal{S}_{\mathbb{F}{}}$ is connected, and that any sufficiently general
linear section of relative dimension $\dim(S)-1$ in a projective embedding
$\mathcal{\bar{S}}\hookrightarrow\mathbb{P}{}_{\mathcal{O}{}_{w}}^{N}{}$ is a
projective flat $\mathcal{O}{}_{w}$-curve $\mathcal{C}{}$ contained in
$\mathcal{S}{}$ with smooth geometrically connected generic fibre (Andr\'{e}
2.5.1). Consider $\left(  \mathcal{X}|\mathcal{C}{}\right)  _{\mathbb{F}{}%
}\rightarrow\mathcal{C}_{\mathbb{F}{}}$. After replacing $\mathcal{C}%
{}_{\mathbb{F}{}}$ by its normalization and pulling back $\left(
\mathcal{X}|\mathcal{C}{}\right)  _{\mathbb{F}{}}$, we are in the situation of
Proposition \ref{g13}. The class $t_{s_{2}}$ is algebraic because the Hodge
conjecture holds for powers of elliptic curves (the $\mathbb{Q}{}$-algebra of
Hodge classes is generated by divisor classes). Hence $(t_{s_{2}\ell})_{0}$ is
algebraic, and \ref{g13} shows that $(t_{s_{1}\ell})_{0}$ is algebraic. This
completes the proof of Proposition \ref{g8}.

\subsection{Absolute Hodge classes on abelian varieties}

Again, $w$ is a prime of $\mathbb{Q}{}^{\mathrm{al}}$ lying over $p$ and
$\ell$ is a prime number $\neq p$.

\begin{theorem}
\label{g43}Assume that the Lefschetz standard conjecture holds for algebraic
varieties over $\mathbb{F}{}$. Let $A$ be an abelian variety over
$\mathbb{Q}{}^{\mathrm{al}}$ with good reduction at $w$ to an abelian variety
$A_{0}$ over $\mathbb{F}{}$, and let $t$ be an absolute Hodge class on $A$.
The class $(t_{\ell})_{0}$ on $A_{0}$ is algebraic.
\end{theorem}

\begin{proof}
We first assume that $A$ is CM, say, of type $(E,\Phi)$. Let $F$ be a
CM-subfield of $\mathbb{C}{}$, finite and Galois over $\mathbb{Q}{}$, that
splits $E$. We may suppose that $F$ contains an imaginary quadratic field in
which $p$ splits.

For each subset $\Delta~$of $\Hom(E,F)$ such that $|t\Delta\cap\Phi
|=r=|t\Delta\cap\bar{\Phi}|$ for all $t\in\Gal(F/\mathbb{Q}{})$, we let
$A_{\Delta}=\prod\nolimits_{s\in\Delta}A\otimes_{E,s}F$. There is an obvious
homomorphism $f_{\Delta}\colon A\rightarrow A_{\Delta}$. The abelian variety
$A_{\Delta}$ is of split Weil type, and every absolute Hodge class $t$ on $A$
can be written as a sum $t=\sum f_{\Delta}^{\ast}(t_{\Delta})$ with
$t_{\Delta}$ a Weil class on $A_{\Delta}$ (\cite{andre1992}; HAV, Theorem 1).
Thus the theorem in this case follows from Proposition \ref{g8}.

We now consider the general case. There exists an abelian scheme $f\colon
X{}\rightarrow S$ over $\mathbb{\mathbb{C}{}}$ with $S$ a connected Shimura
variety, and a section $\gamma$ of $R^{2r}f_{\ast}\mathbb{A}{}$ such that
$(X,\gamma)_{s}=(A,t)$ (\cite{deligne1982}, 6.1). As before, we may suppose
that $f$ is defined over $\mathbb{Q}{}^{\mathrm{al}}$ and that $s\in
S(\mathbb{Q}{}^{\mathrm{al}})$. There exists a point $s^{\prime}\in
S(\mathbb{Q}{}^{\mathrm{al}})$ such that $(s^{\prime})_{0}=s_{0}$ in
$S_{0}(\mathbb{F}{})$ and $X_{s^{\prime}}$ is a CM abelian variety (Kisin,
Vasiu). Now the theorem for $X_{s^{\prime}}$ implies that $(t_{s\ell})_{0}$ is algebraic.
\end{proof}

\section{Proof of Theorem 2.}

Fix an algebraic closure $\mathbb{F}{}$ of $\mathbb{F}{}_{p}$, and let
$\mathbb{F}{}_{q}$ be the subfield of $\mathbb{F}{}$ with $q$ elements.

\begin{plain}
\label{g17}Let $X$ be an algebraic variety over $\mathbb{F}{}_{q}$. For
$\ell\neq p$, the Tate conjecture $T(X,\ell)$ states that the $\mathbb{Q}%
{}_{\ell}$-vector space $H_{\ell}^{2\ast}(X)(\ast)^{\Gal(\mathbb{F}%
{}/\mathbb{F}{}_{q})}$ is spanned by algebraic classes, and the conjecture
$S(X,\ell)$ states that the obvious map $H_{\ell}^{2\ast}(X)(\ast
)^{\Gal(\mathbb{F}{}/\mathbb{F}{}_{q})}\rightarrow H_{\ell}^{2\ast}%
(X)(\ast)_{\Gal(\mathbb{F}/\mathbb{F}{}_{q})}$ is an isomorphism. The full
Tate conjecture $T(X)$ states that, for all $r$, the pole of the zeta function
$Z(X,t)$ at $t=q^{-r}$ is equal to the rank of the group of numerical
equivalence classes of algebraic cycles on $X$ of codimension $r$. It is known
(folklore) that, if $T(X,\ell)$ and $S(X,\ell)$ hold for a single $\ell$, then
the full Tate conjecture $T(X)$ holds, in which case $T(X,\ell)$ and
$S(X,\ell)$ hold for all $\ell$. See \cite{tate1994}.

We say that one of these conjectures holds for an algebraic variety $X$ over
$\mathbb{F}{}$ if it holds for all models of $X$ over finite subfields of
$\mathbb{F}{}$ (it suffices to check that it holds for some model over a
sufficiently large subfield).
\end{plain}

\begin{theorem}
\label{g40}Assume that the Lefschetz standard conjecture holds for algebraic
varieties over $\mathbb{F}{}$ and $\ell$-adic \'{e}tale cohomology (some
$\ell\neq p$). Then the full Tate conjecture holds for abelian varieties over
finite fields of characteristic $p$.
\end{theorem}

\begin{proof}
In \cite{milne1999lm}, the Tate conjecture for abelian varieties over
$\mathbb{F}{}$ is shown to follow from the Hodge conjecture for CM abelian
varieties over $\mathbb{C}{}$. However, the proof does not use that the Hodge
classes are algebraic, but only that they become algebraic modulo $p$. Hence
we can deduce from Proposition \ref{g43} that the Tate conjecture holds for
abelian varieties over $\mathbb{F}{}$ and some $\ell$. As the Frobenius map
acts semisimply on the cohomology of abelian varieties (Weil 1948), this
implies that the full Tate conjecture holds for abelian varieties over
$\mathbb{F}{}$.
\end{proof}

\begin{theorem}
\label{g41}Assume that the Lefschetz standard conjecture holds for algebraic
varieties over $\mathbb{F}{}$ and $\ell$-adic \'{e}tale cohomology (some
$\ell\neq p$). Then Grothendieck's standard conjecture of Hodge type holds for
abelian varieties over fields of characteristic $p$ and the classical Weil
cohomology theories.
\end{theorem}

\begin{proof}
In \cite{milne2002p} the Hodge standard conjecture for abelian varieties in
characteristic $p$ is shown to follow from the Hodge conjecture for CM abelian
varieties over $\mathbb{C}{}$. Again, the proof uses only that the Hodge
classes become algebraic modulo $p$, and so the theorem follows from
Proposition \ref{g43}.
\end{proof}

\begin{corollary}
\label{g42}Assume that the Lefschetz standard conjecture holds for algebraic
varieties over $\mathbb{F}{}$ and $\ell$-adic \'{e}tale cohomology (some
$\ell\neq p$). Then the conjecture of Langlands and Rapoport (1987,
5.e)\nocite{langlandsR1987} is true for simple Shimura varieties of PEL-types
A and C.
\end{corollary}

\begin{proof}
Langlands and Rapoport (ibid., \S 6) prove this under the assumption of the
Hodge conjecture for CM abelian varieties and the Tate and Hodge standard
conjectures for abelian varieties over $\mathbb{F}{}$. However, their argument
does not use that Hodge classes on CM abelian varieties are algebraic, but
only that they become algebraic modulo $p$. As this, together with the Tate
and Hodge standard conjectures, are implied by the Lefschetz standard
conjecture, so also is their conjecture.
\end{proof}

\section{Proof of Theorem 3.}

Briefly, the Tate conjecture over $\mathbb{F}{}$ implies the Lefschetz
standard conjecture over $\mathbb{F}{}$, and hence the Hodge standard
conjecture for abelian varieties (Theorem \ref{g3}). Now form the category of
abelian motives over $\mathbb{F}{}$: Grothendieck's standard conjectues hold
for it. The full Tate conjecture implies that the category of abelian motives
contains the motives of all algebraic varieties over $\mathbb{F}{}$, and so
the Hodge standard conjecture holds for them also.

We now prove more precise statements.

\begin{proposition}
\label{g29}Let $X$ be an algebraic variety over $\mathbb{F}{}$. If the Tate
conjecture holds for $X$ and some $\ell$, then the Lefschetz standard
conjecture holds for $X$ and the same $\ell$.
\end{proposition}

\begin{proof}
To prove the Lefschetz standard conjecture for $X$ and a prime $\ell$, it
suffices to show that, for each $i\leq d\overset{\df}{=}\dim(X)$, there exists
an algebraic correspondence inducing an isomorphism $H_{\ell}^{2d-i}%
(X)\rightarrow H_{\ell}^{i}(X)$ (\cite{kleiman1994}, 4-1, $\nu
(X)\Leftrightarrow B(X)$). The inverse $\theta^{i}$ of the Lefschetz map
$L^{d-i}\colon H_{\ell}^{i}(X)\rightarrow H_{\ell}^{2d-i}(X)(d-i)$ is an
isomorphism $H_{\ell}^{2d-i}(X)(d-i)\rightarrow H_{\ell}^{i}(X)$ commuting
with the action of the Galois group. Any algebraic class $\nu^{i}$
sufficiently close to the graph of $\theta^{i}$ will induce the required isomorphism.
\end{proof}

\begin{proposition}
\label{g20}Let $H$ be a Weil cohomology theory on algebraic varieties over an
algebraically closed field $k$, and let $X$ and $Y$ be algebraic varieties
over $k$. Assume that there exists an algebraic correspondence $\alpha$ on
$X\times Y$ such that%
\[
\alpha_{\ast}\colon H^{\ast}(X)\rightarrow H^{\ast}(Y)
\]
is injective. If the Hodge standard conjecture holds for $Y$, then it holds
for $X$.
\end{proposition}

\begin{proof}
Apply \cite{kleiman1968}, 3.11, and \cite{saavedra1972}, VI, 4.4.2.
\end{proof}

\begin{lemma}
\label{g21}Let $X$ be an algebraic variety over $\mathbb{F}{}_{q}$. If
$S(X\times X,\ell)$ holds for some $\ell$, then the Frobenius endomorphism
acts semisimply on the $\ell$-adic \'{e}tale cohomology of $X$.
\end{lemma}

\begin{proof}
The statement $S(X\times X,\ell)$ says that $1$, if an eigenvalue of the
Frobenius element acting on the $\ell$-adic cohomology of $X\times X$, is
semisimple. From the K\"{u}nneth formula%
\[
H_{\ell}^{r}(X\times X)\simeq\bigoplus\nolimits_{i+j=r}H_{\ell}^{i}(X)\otimes
H_{\ell}^{j}(X)
\]
and linear algebra, we see that this implies that all eigenvalues on $H_{\ell
}^{\ast}(X)$ are semisimple.
\end{proof}

It is conjectured that the Frobenius element always acts semisimply
(Semisimplicity Conjecture).

Fix a power $q$ of $p$ and a prime $\ell\neq p$. Define a \emph{Tate structure
}to be a finite-dimensional $\mathbb{Q}{}_{\ell}$-vector space with a linear
(Frobenius) map $\varpi$ whose characteristic polynomial lies in $\mathbb{Q}%
{}[T]$ and whose eigenvalues are Weil $q$-numbers, i.e., algebraic numbers
$\alpha$ such that, for some integer $m$ (called the weight of $\alpha$),
$\left\vert \rho(\alpha)\right\vert =q^{m/2}$ for every homomorphism
$\rho\colon\mathbb{Q}{}[\alpha]\rightarrow\mathbb{C}$, and, for some integer
$n$, $q^{n}\alpha$ is an algebraic integer. \noindent\ When the eigenvalues
are all of weight $m$ (resp. algebraic integers, resp. semisimple), we say
that $V$ is of \emph{weight} $m$ (resp. \emph{effective}, resp.
\emph{semisimple}). For example, for any smooth complete variety $X$ over $k$,
$H_{\ell}^{i}(X)$ is an effective Tate structure of weight $i/2$
(\cite{deligne1974}), which is semisimple if $X$ is an abelian variety
(\cite{weil1948}, no.~70).

\begin{proposition}
\label{g23}Every effective semisimple Tate structure is isomorphic to a Tate
substructure of $H_{\ell}^{\ast}(A)$ for some abelian variety $A$ over
$\mathbb{F}{}_{q}$.
\end{proposition}

\begin{proof}
We may assume that the Tate structure $V$ is simple. Then $V$ has weight $m$
for some $m\geq0$, and the characteristic polynomial $P(T)$ of $\varpi$ is a
monic irreducible polynomial with coefficients in $\mathbb{Z}{}$ whose roots
all have real absolute value $q^{m/2}$. According to Honda's theorem
(\cite{honda1968}; \cite{tate1968}), $P(T)$ is the characteristic polynomial
of an abelian variety $A$ over $\mathbb{F}{}_{q^{m}}$. Let $B$ be the abelian
variety over $\mathbb{F}{}_{q}$ obtained from $A$ by restriction of the base
field. The eigenvalues of the Frobenius map on $H_{\ell}^{1}(B)$ are the
$m\mathrm{th}$-roots of the eigenvalues of the Frobenius map on $H_{\ell}%
^{1}(A)$, and it follows that $V$ is a Tate substructure of $H_{\ell}^{m}(B)$.
\end{proof}

\begin{theorem}
\label{g26}Let $X$ be an algebraic variety over $\mathbb{F}{}$, and let $\ell$
be a prime $\neq p$. If the Frobenius map acts semisimply on $H_{\ell}^{\ast
}(X)$ and the Tate conjecture holds for $\ell$ and all varieties of the form
$X\times A$ with $A$ an abelian variety, then the Hodge standard conjecture
holds for $X$ and $\ell$.
\end{theorem}

\begin{proof}
According to \ref{g23}, there exists an inclusion $H_{\ell}^{\ast
}(X)\hookrightarrow H_{\ell}^{\ast}(A)$ of Tate structures with $A$ an abelian
variety. This map is defined by a cohomological correspondence on $X\times A$
fixed by the Galois group. Any algebraic correspondence sufficiently close to
this correspondence defines an inclusion $H_{\ell}^{\ast}(X)\hookrightarrow
H_{\ell}^{\ast}(A)$. Now we can apply Proposition \ref{g20}.
\end{proof}

\begin{corollary}
\label{g28}If the Tate and semisimplicity conjectures hold for all algebraic
varieties over $\mathbb{F}{}$ and some prime number $\ell$, then both the full
Tate and Grothendieck standard conjectures hold for all algebraic varieties
over $\mathbb{F}{}$ and all $\ell$.
\end{corollary}

\begin{proof}
Immediate consquence of the theorem.
\end{proof}

\section{An unconditional variant}

We use \cite{schappi2020} to replace some of the above statements by
unconditional variants.

\subsection{Characteristic zero}

Let $k$ be an algebraically closed field of characteristic zero, and fix an
embedding $k\hookrightarrow\mathbb{C}{}$. Let $H$ denote the Weil cohomology
theory $X\rightsquigarrow H^{\ast}(X(\mathbb{C}{}),\mathbb{Q}{})$, and let
$\Mot_{H}(k)$ denote the category of motives defined using almost-algebraic
classes as correspondences. It is a graded pseudo-abelian rigid tensor
category\footnote{tensor category (functor)\,=\,symmetric monoidal category
(functor)} over $\mathbb{Q}{}$.

According to \cite{schappi2020}, \S 3, the fibre functor $\omega_{H}%
\colon\Mot_{H}(k)\rightarrow\mathbb{Z}{}$-$\Vc_{\mathbb{Q}{}}$ factors in a
canonical way through a \textquotedblleft universal\textquotedblright\ graded
tannakian category $\mathcal{M}{}_{H}(k)$ over $\mathbb{Q}{}$,%
\[
\Mot_{H}(k)\overset{[-]}{\longrightarrow}\mathcal{M}{}_{H}(k)\overset{\omega
}{\longrightarrow}\mathbb{Z}{}\text{-}\Vc_{\mathbb{Q}{}},
\]
where $\omega$ is a graded fibre functor.\footnote{fibre functor$\,=\,$exact
faithful tensor functor}

We define the \emph{algebraic*} classes on an algebraic variety $X$ over $k$
to be the elements of $\Hom(\1,[h(X)])$. The Lefschetz standard conjecture
holds for algebraic* classes (\cite{schappi2020}, \S 3; alternatively, apply
Corollary \ref{g78} below).

Now $\omega_{H}$ is a functor from $\Mot_{H}(k)$ into the category
$\Hdg_{\mathbb{Q}{}}$ of polarizable rational Hodge structures. This factors
through $\mathcal{M}{}_{H}(k)$,%
\[
\Mot_{H}(k)\overset{[-]}{\longrightarrow}\mathcal{M}{}_{H}(k)\overset{\omega
}{\longrightarrow}\Hdg_{\mathbb{Q}{}}\text{,}%
\]
where $\omega$ is a functor of graded tannakian categories. Therefore
algebraic* classes on $X$ are Hodge classes relative to the given embedding of
$k$ into $\mathbb{C}{}$. It follows that Grothendieck's standard conjecture of
Hodge type holds for algebraic* classes. Moreover, all algebraic* classes on
abelian varieties are absolutely Hodge (\cite{deligne1982}, 2.11).

The same proof as for almost-algebraic classes (see \S 1) shows that the Hodge
conjecture holds for algebraic* classes on abelian varieties over
$\mathbb{C}{}$, i.e., all Hodge classes on abelian varieties over
$\mathbb{C}{}$ are algebraic*. As a consequence, for abelian varieties
satisfying the Mumford-Tate conjecture, the Tate conjecture holds for
algebraic* classes.

\subsection{Characteristic $p$}

Fix a prime number $p$, and let $\mathbb{F}{}$ denote an algebraic closure of
$\mathbb{F}{}_{p}$. For $\ell\neq p$, we let $\Mot_{\ell}(\mathbb{F}{})$
denote the category of motives over $\mathbb{F}{}$ defined using algebraic
classes modulo $\ell$-adic homological equivalence as correspondences. It is a
graded pseudo-abelian rigid tensor category over $\mathbb{Q}{}$.

According to \cite{schappi2020}, \S 3, the graded tensor functor $\omega
_{\ell}\colon\Mot_{\ell}(\mathbb{F}{})\rightarrow\mathbb{Z}{}$%
-$\Vc_{\mathbb{Q}{}_{\ell}}$ factors in a canonical way through a graded
tannakian category $\mathcal{M}{}_{\ell}(\mathbb{F}{})$,%
\[
\Mot_{\ell}(k)\overset{[-]}{\longrightarrow}\mathcal{M}{}_{\ell}(\mathbb{F}%
{})\overset{\omega}{\longrightarrow}\mathbb{Z}{}\text{-}\Vc_{\mathbb{Q}{}},
\]
where $\omega$ is a graded fibre functor. Unfortunately, we do not know that
$\End(\1)=\mathbb{Q}{}$ in $\mathcal{M}{}_{\ell}(\mathbb{F}{})$, only that it
is a subfield of $\mathbb{Q}{}_{\ell}.$\footnote{Andr\'{e}'s category of
motivated classes in characteristic $p$ has the same problem.}

Let $X$ be an algebraic variety over $\mathbb{F}$. We define the
\textbf{algebraic*} classes on $X$ to be the elements of $\Hom(\1,[h(X)])$. As
before, the Lefschetz standard conjecture holds for algebraic* classes.
Therefore Proposition \ref{g13} holds unconditionally for algebraic* classes:
let $f\colon X\rightarrow S$ be as in the proposition, and let $t$ be a global
section of the sheaf $R^{2r}f_{\ast}\mathbb{Q}{}_{\ell}(r)$; if $t_{s}$ is
algebraic* for one $s\in S(\mathbb{F}{})$, then it is algebraic* for all $s$.

\begin{remark}
\label{g30}Until it is shown that $\End(\1)=\mathbb{Q}$ in $\mathcal{M}%
{}_{\ell}(\mathbb{F}{})$, this category is of only modest interest. For
abelian motives, what is needed is a proof of the rationality conjecture
(\cite{milne2009}, 4.1).\footnote{Let $A$ be an abelian variety over
$\mathbb{Q}{}^{\mathrm{al}}$ with good reduction to an abelian variety $A_{0}$
over $\mathbb{F}{}$; the cup product of the specialization to $A_{0}$ of any
absolute Hodge class on $A$ with a product of divisors of complementary
codimension lies in $\mathbb{Q}{}$ .}
\end{remark}

\subsection{Mixed characteristic}

Fix a prime $w$ of $\mathbb{Q}{}^{\mathrm{al}}$ dividing $p{}$ and a prime
number $\ell\neq p$. Theorem \ref{g43} holds unconditionally for algebraic*
classes: let $A$ be an abelian variety over $\mathbb{Q}{}^{\mathrm{al}}$ with
good reduction at $w$ to an abelian variety $A_{0}$ over $\mathbb{F}{}$, and
let $t$ be an absolute Hodge class (e.g., an algebraic* class) on $A$; then
$(t_{\ell})_{0}$ is an algebraic* class on $A_{0}$. The proof is the same as
before, using the * version of Proposition \ref{g13}.

We deduce, as in the proof of Theorem \ref{g40}, that the Tate conjecture
holds for algebraic* classes on abelian varieties over $\mathbb{F}{}$, i.e.,
that $\ell$-adic Tate classes on abelian varieties over $\mathbb{F}{}$ are algebraic*.

Let $\mathcal{M}{}_{H}^{\prime}(\mathbb{Q}{}^{\mathrm{al}})$ denote the
tannakian subcategory of $\mathcal{M}{}_{H}(\mathbb{Q}{}^{\mathrm{al}})$
generated by abelian varieties with good reduction at $w$. There is a
canonical tensor functor $\mathcal{M}{}_{H}^{\prime}(\mathbb{Q}{}%
^{\mathrm{al}})\rightarrow\mathcal{M}{}_{\ell}(\mathbb{F}{})$.

\section{Statements implying the Lefschetz standard conjecture}

\subsection{Conjecture $D$ and the Lefschetz standard conjecture}

Let $H$ be a Weil cohomology theory. The next statement goes back to Grothendieck.

\begin{proposition}
\label{g67}Assume that $H$ satisfies the strong Lefschetz theorem. Conjecture
$D(X)$ implies $A(X,L)$ (all $L)$; in the presence of the Hodge standard
conjecture, $A(X,L)$ (one $L$) implies $D(X)$.
\end{proposition}

\begin{proof}
Conjecture $D(X)$ says that the pairing
\begin{equation}
x,y\mapsto\langle x\cdot y\rangle\colon A_{H}^{i}(X)\times A_{H}%
^{d-i}(X)\rightarrow A_{H}^{d}(X)\simeq\mathbb{Q}{} \label{e1}%
\end{equation}
is nondegenerate for all $i\leq d\overset{\df}{=}\dim(X)$. Therefore, $\dim
A_{H}^{i}(X)=\dim A_{H}^{d-i}(X)$. As the map $L^{d-2i}\colon A_{H}%
^{i}(X)\rightarrow A_{H}^{d-i}(X)$ is injective, it is surjective, i.e.,
$A(X,L)$ holds. The converse is equally obvious.
\end{proof}

\begin{corollary}
\label{g68}Conjecture $D(X\times X)$ implies $B(X)$.
\end{corollary}

\begin{proof}
Indeed, $A(X\times X,L\otimes1+1\otimes L)$ implies $B(X)$ (\cite{kleiman1968}%
, Theorem 4-1).
\end{proof}

\begin{remark}
\label{g75}If Conjecture $D(X\times X)$ holds whenever $X$ is an abelian
scheme over a complete smooth curve over $\mathbb{C}{}$, then the Hodge
conjecture holds for abelian varieties.
\end{remark}

\subsection{Does Conjecture C imply Conjecture B?}

Kleiman (1994) states eight versions of Grothendieck's standard conjecture of
Lefschetz type. He proves that six of the eight are equivalent and that a
seventh is \textquotedblleft practically equivalent\textquotedblright\ to the
others, but he states that the eighth version, Conjecture C, \textquotedblleft
is, doubtless, truly weaker\textquotedblright. In this subsection we examine
whether Conjecture C is, in fact, equivalent to the remaining conjectures.

Let $H$ be a Weil cohomology theory on the algebraic varieties over an
algebraically closed field $k$. Assume that $H$ satisfies conjecture $C$, and
let $\Mot_{H}(k)$ denote the category of motives defined using algebraic
classes modulo homological equivalence as the correspondences. It is a graded
pseudo-abelian rigid tensor category over $\mathbb{Q}{}$ equipped with a
graded tensor functor $\omega_{H}\colon\Mot_{H}\rightarrow\mathbb{Z}{}%
$-$\Vc_{Q{}}$, where $Q$ is the coefficient field of $H$.

\begin{proposition}
\label{g77} Assume that $H$ satisfies the strong Lefschetz theorem in addition
to Conjecture $C$. If $\omega_{H}$ is conservative, then $H$ satisfies the
Lefschetz standard conjecture.
\end{proposition}

\begin{proof}
Let $L\colon H^{r}(X)\rightarrow H^{r+2}(X)(1)$ be the Lefschetz operator
defined by a hyperplane section of $X$. By assumption%
\begin{equation}
L^{d-2i}\colon H^{2i}(X)(i)\rightarrow H^{2d-2i}(X)(d-i) \label{eq2}%
\end{equation}
is an isomorphism for all $2i\leq d\overset{\df}{=}\dim(X)$. As $\omega_{H}$
is conservative,%
\begin{equation}
l^{d-2i}\colon h^{2i}(X)(i)\rightarrow h^{2d-2i}(X)(d-i) \label{eq3}%
\end{equation}
is an isomorphism for all $2i\leq d$. On applying the functor $\Hom(\1,-)$ to
this isomorphism, we get an isomorphism%
\[
L^{d-2i}\colon A_{H}^{i}(X)\rightarrow A_{H}^{d-i}(X).
\]
Thus, Conjecture $A(X,L)$ is true.
\end{proof}

\begin{corollary}
\label{g78}Assume that $H$ satisfies the strong Lefschetz theorem and
Conjecture $C$. If $\Mot_{H}(k)$ is tannakian, then $H$ satisfies Conjecture
$B$.
\end{corollary}

\begin{proof}
Fibre functors on tannakian categories are conservative.
\end{proof}

Proposition \ref{g67} shows that a Weil cohomology theory satisfying both the
strong Lefschetz theorem and Conjecture $D$ also satisfies Conjecture $B$.
Here we prove a stronger result.

\begin{proposition}
\label{g69}Suppose that there exists a Weil cohomology theory $\mathcal{H}{}$
satisfying both the strong Lefschetz theorem and Conjecture $D$. Then every
Weil cohomology theory $H$ satisfying the strong Lefschetz theorem and
Conjecture $C$ also satisfies Conjecture $B$.
\end{proposition}

\begin{proof}
Let $\mathcal{H}{}$ and $H$ be Weil cohomology theories satisfying the strong
Lefschetz theorem and assume that $\mathcal{H}{}$ (resp.$~H$) satisfies
Conjecture $D$ (resp.~Conjecture $C$). Then $\mathcal{H}{}$ satisfies the
Lefschetz conjecture (\ref{g67}), in particular, Conjecture $C$. Let
$\Mot_{\mathrm{num}}(k)=\Mot_{\mathcal{H}{}}(k)$ be the category of motives
defined using algebraic cycles modulo numerical equivalence as
correspondences. Then $\Mot_{\mathrm{num}}$ is a semisimple tannakian category
over $\mathbb{Q}{}$ (Jannsen, Deligne), and there is a quotient functor
$q\colon\Mot_{H}\rightarrow\Mot_{\mathrm{num}}$. For each $M$ in $\Mot_{H}$,
the map $\End(M)\rightarrow\End(qM)$ is surjective with kernel the radical of
the ring $\End(M)$, and this radical is nilpotent (\cite{jannsen1992}).

The conditions on $\mathcal{H}{}$ imply that it satisfies Conjecture $B$
(Proposition \ref{g67}). This means that for each $i\leq d\overset{\df}{=}%
\dim(X)$, there exists a morphism $h_{\mathrm{num}}^{2d-i}(X)(d-i)\rightarrow
h_{\mathrm{num}}^{i}(X)$ inducing the inverse of the map%
\[
L^{d-i}\colon\mathcal{H}{}_{\mathrm{num}}^{i}(X)\rightarrow\mathcal{H}%
{}_{\mathrm{num}}^{2d-i}(X)(d-i)\text{.}%
\]

Write $\alpha$ for the morphism $h^{i}(X)\rightarrow h^{2d-i}(X)(d-i)$ in
$\Mot_{H}(k)$ inducing the isomorphism
\begin{equation}
L^{d-i}\colon H^{i}(X)\rightarrow H^{2d-i}(X)(d-i). \label{eq4}%
\end{equation}
According to the last paragraph, there exists a morphism $\beta\colon
h^{2d-i}(X)(d-i)\rightarrow h^{i}(X)$ such that $q(\beta\circ\alpha
)=\id_{h_{\mathrm{num}}^{i}(X)}$. Now $\beta\circ\alpha=1+n$ in $\End(h^{i}%
(X))$, where $n$ is nilpotent. On replacing $\beta$ with $(1-n+n^{2}%
-\cdots)\circ\beta$, we find that $\beta\circ\alpha=1$ in $\End(h^{i}(X))$.
Hence the inverse of the map (\ref{eq4}) is algebraic, as required.
\end{proof}

\begin{proposition}
\label{g71}If there exists one Weil cohomology theory satisfying the strong
Lefschetz theorem and Conjecture $D$, then every Weil cohomology theory
satisfying Conjecture $D$ also satisfies the strong Lefschetz theorem
\end{proposition}

\begin{proof}
If there exists a Weil cohomology theory satisfying the strong Lefschetz
theorem and Conjecture $D$, then in $\Mot_{\mathrm{num}}(k)$,
\[
l^{d-i}\colon h^{i}(X)\rightarrow h^{2d-i}(X)(d-i)
\]
is an isomorphism for $i\leq d$. Let $H$ be a Weil cohomology theory
satisfying Conjecture $D$. On applying $H$ to this isomorphism, we get an
isomorphism%
\[
L^{d-i}\colon H^{i}(X)\rightarrow H^{2n-r}(X)(n-r).
\]

\end{proof}

\begin{remark}
\label{g72}Because $\Mot_{\mathrm{num}}$ is Tannakian, there exists a field
$Q$ of characteristic zero ${}$ and a $Q$-valued fibre functor $\omega$. Then
$\mathcal{H}{}\colon X\rightsquigarrow\bigoplus_{i}\omega(X,\pi_{i},0)$ is a
Weil cohomology theory satisfying Conjecture $D$. It remains to show that
$\omega$ can be chosen so that $\mathcal{H}{}$ satisfies the strong Lefschetz
theorem. This comes down to showing that $l^{d-i}\colon h^{i}(X)\rightarrow
h^{2d-i}(X)(d-i)$ is an isomorphism in $\Mot_{\mathrm{num}}(k)$.
\end{remark}

\begin{remark}
\label{g74}Every Weil cohomology theory satisfying the weak Lefschetz theorem
also satisfies the strong Lefschetz theorem (Katz and Messing 1974,
Corollaries to Theorem~1).
\end{remark}

\bibliographystyle{cbe}
\bibliography{D:/Current/refs}

\end{document}